\documentclass[12pt]{article}

\usepackage{amsmath,amsfonts,amssymb,amsthm,graphicx,float,color,caption,verbatim,subfig}
\usepackage{geometry}
\theoremstyle{plain}% default
\newtheorem*{The unknotting conjecture}{The unknotting conjecture}
\newtheorem{thm}{Theorem}[section]
\newtheorem{lem}[thm]{Lemma}
\newtheorem{prop}[thm]{Proposition}
\newtheorem{cor}[thm]{Corollary}

\theoremstyle{remark}
\newtheorem{defn}[thm]{Definition}

\newtheorem{que}[thm]{Question}
\newtheorem*{Acknowledgement}{Acknowledgement}

\theoremstyle{remark}

\title{The triple point number of surface-knots of genus one is
at least four}
\author{A. Al Kharusi and T. Yashiro   }
\date{}
\begin{document}
\maketitle
\begin{abstract}
In this paper, we show that there is no surface-knot of genus one with triple point
number invariant equal to three.
\end{abstract}
\section{Introduction}
A \textit{surface-knot} $F$ is a connected closed orientable surface embedded smoothly in the Euclidean $4$-space $\mathbb{R}^4$. It is called a \textit{2-knot} if it is homeomorphic to a 2-sphere. A surface-knot $F$ is called \textit{trivial} or \textit{unknotted} if $F$ bounds a handlebody embedded in $\mathbb{R}^3 \times \{0\}$. 
Projecting a surface-knot to $\mathbb{R}^3$ gives a \textit{generic surface} with singularity point set consisting of double points, isolated triple points and isolated branch points. 
The \textit{triple point number} of a surface-knot $F$, denoted by $t(F)$, is the minimal number of triple points taken over all possible projections of $F$.  Knotted surface-knots with zero triple point number form a well-known family of surface-knots, called \textit{pseudo ribbon}\cite{paper_2Kaw} (it is called \textit{ribbon 2-knots} in case of 2-knots).  It is known that there is no surface-knot $F$ with $t(F)=1$ \cite{paper_2Sat}. S. Satoh in \cite{paper_2Sat3} showed that for 2-knots, the triple point number is at least four.  In particular, the 2-twist spun trefoil is a 2-knot with the triple point number equal to four \cite{paper_2SatShim,paper_2Sat,paper_2Sat3}. By attaching a trivial 1-handle to the 2-twist spun trefoil, we obtain a surface-knot $F$ of genus one satisfying $t(F)=4$.  The authors proved in \cite{amtsu1} that there is no genus-one surface-knot with triple point number equal to two. The aim of this paper is to answer the following question.
\begin{que}\label{d}
Is there a surface-knot $F$ of genus one with $t(F)=3$?
\end{que} 
The paper is organized as follows. Sections 2 to 5 recall some basics about surface-knot diagrams including double point curves, double decker sets,  Alexander numbering of triple points and Roseman moves. Section 6 describes the diagram with three triple points of
a surface-knot $F$ satisfying $t(F) = 3$. Section 7 introduces some lemmas. Section 8 gives the
answer to the Question \ref{d}, where we show that there is no surface-knot of genus one with
triple point number equal to three.
\section{Double point curves of surface-knot diagrams}
Let $\Delta$ be a surface-knot diagram of a surface-knot $F$. By connecting double edges which are in opposition to each other at each triple point of $\Delta$,  the singularity set of a projection is regarded as a union of curves (circles and arc components) immersed in 3-space. We call such an oriented curve a \textit{double point curve}. A  double point curve with no triple points is called \textit{trivial}.

\begin{lem}[\cite{paper_2BW}]\label{even}
The number of triple points along each double point circle is even.
\end{lem}
\begin{proof}
We assign a BW orientation to the singularity set such that the orientation restricted to double edges at every triple point is as depicted in Figure \ref{ori} (see  \cite{paper_2BW} for the BW orientation). It follows that a pair of double edges that are in opposition to each other at a triple point $T$ admit opposite orientations on both sides of $T$. Hence $n$ must be even.
\end{proof} 
 \begin{figure}[H]
\centering
\captionsetup{font=scriptsize}      
\mbox{\includegraphics[scale=0.66]{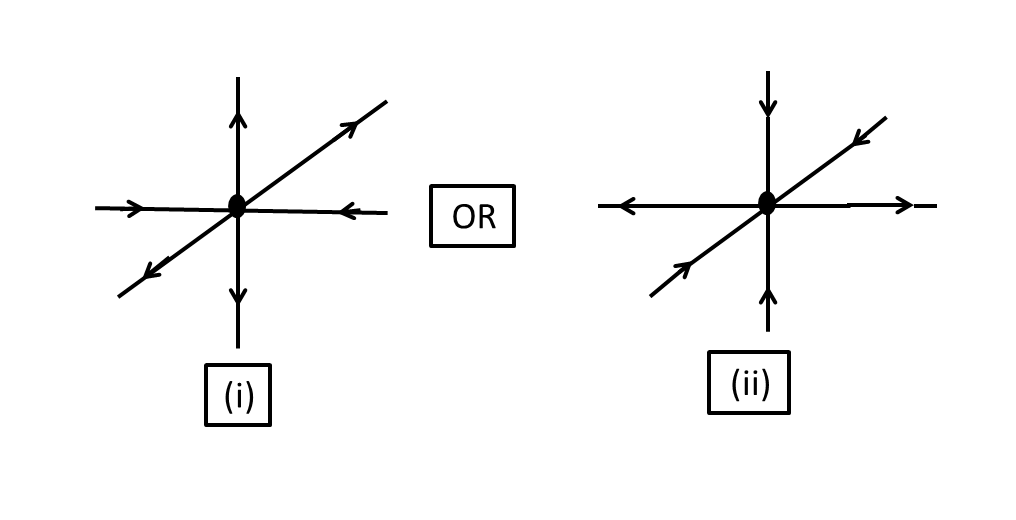}}
 \caption{BW orientation to the double edges at the triple point}
\label{ori}
\end{figure}

\section{Double decker sets of surface-knot diagrams}
The closure of the preimages of the double and triple point set of the projection of a surface-knot is called a \textit{double decker set} \cite{CarterSaito}.  There are two sheets intersect transversally at a double edge which we call the \textit{upper} and the \textit{lower} sheets with respect to the projection. The preimage of a double edge $e$  consists of two open arcs $e^U$ and  $e^L$ that lie in the preimages of upper and lower sheets, respectively. We call the preimage of a branch point also a \textit{branch point} and it is included in the closure of the preimage of the double point set. There are three sheets intersect transversally at a triple point $T$ which we call the \textit{top}, the \textit{middle} and the \textit{bottom} sheets with respect to the projection. The preimage of a triple point, $T$, is a union of three transverse crossing points $T^T$, $T^M$ and $T^B$ that lie in the preimages of the top, the middle and the bottom sheets, respectively.  
\begin{defn}
Let $C$ be a double point curve of $\Delta$. Then the union $C^U=\bigcup_{{e} \in C} \big(\overline{{e}^U}
\big)$ is called an \textit{upper decker curve}, where $ \overline{{e}^U}$ stands for the closure of $e^U$. The \textit{upper decker set} is the set $\mathcal{C}^U=\{C_1^U, \dots , C_n^U\}$ of upper decker curves.
\end{defn}
A \textit{lower decker curve} $C^L$ and the lower decker set $\mathcal{C}^L$ are defined analogously. In fact, the upper or lower decker curve is a circle or arc component immersed into $F$. The transversal crossing points of the immersed curves correspond to triple points in the projection.

\section{Signs of triple points, Alexander numbering and orientation of double edges}
We give sign to the triple point  $T$ of a surface-knot diagram as follows. Let $n_T$, $n_M$ and $n_B$ denote the normal vectors to the top, the middle and the bottom sheets presenting their orientations, respectively. The \textit{sign} of  $T$, denoted by $\epsilon(T)$, is $+1$ if the triplet $(n_T,n_M,n_B)$ matches the orientation of $\mathbb{R}^3$ and otherwise $-1$. This definition can be found in \cite{3}.\\
An \textit{Alexander numbering} for a surface-knot is a function that associates to  each 3-dimensional complementary region of the diagram an integer called the \textit{index} of the region.  These integers are determined by the following:
\begin{itemize}
\item[(i)] the indices of regions that are separated by a sheet differ by one and
\item[(ii)] the region into which a normal vector to the sheet points has the larger index. 
\end{itemize}
For a triple point $T$, the integer $\lambda(T)$ is called the \textit{Alexander numbering} of $T$ and defined as  the minimal Alexander index among the eight regions surrounding $T$.  Equivalently, $\lambda(T)$ is the Alexander index of a specific region $R$, called a \textit{source region},  where all orientation normals to the bounded sheets point away from $R$.\\
We define an orientation to a double edge in a surface-knot diagram so that for a tangent vector $v$ to the double edge at a double point $\mathcal{D}$, the triple $(n_U,n_L,v)$ gives the orientation of $\mathbb{R}^3$, where $n_U$ and $n_L$ are normal vectors to the upper sheet and the lower sheet at $\mathcal{D}$ presenting their orientations, respectively.

\section{Roseman moves}
Two surface-knot diagrams are said to be equivalent if and only if they are related by a finite sequence of seven local moves called \textit{Roseman moves}.  The move from left to right is denoted by $R-i^+$ and from right to left by $R-i^-$ $(i=1,2,3,4,5,6)$.
\begin{figure}[H]
\centering
\captionsetup{font=scriptsize}      
\mbox{\includegraphics[scale=0.4]{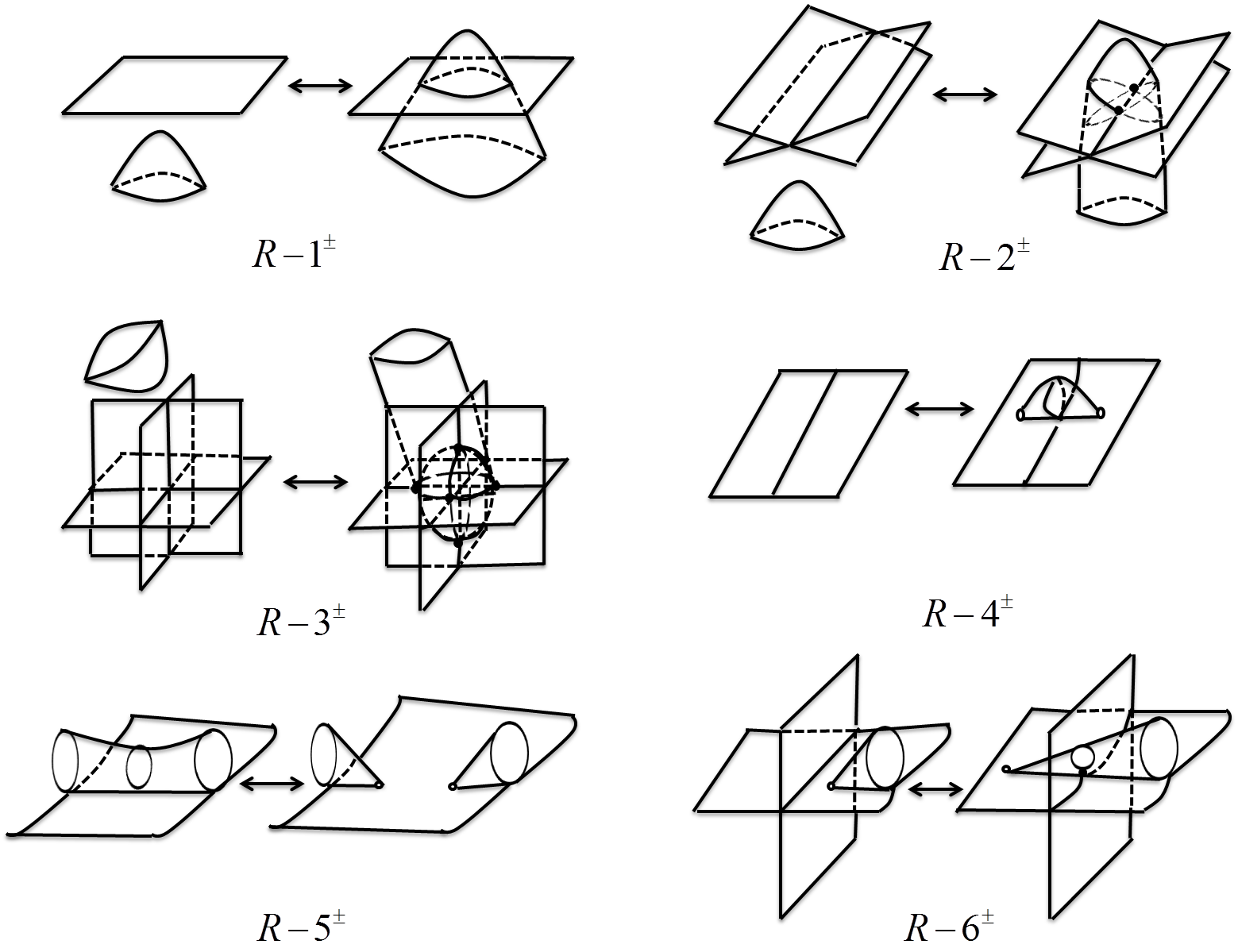}}
\end{figure}
\begin{figure}[H]
\centering
\captionsetup{font=scriptsize}       
   \mbox{\includegraphics[scale=0.2]{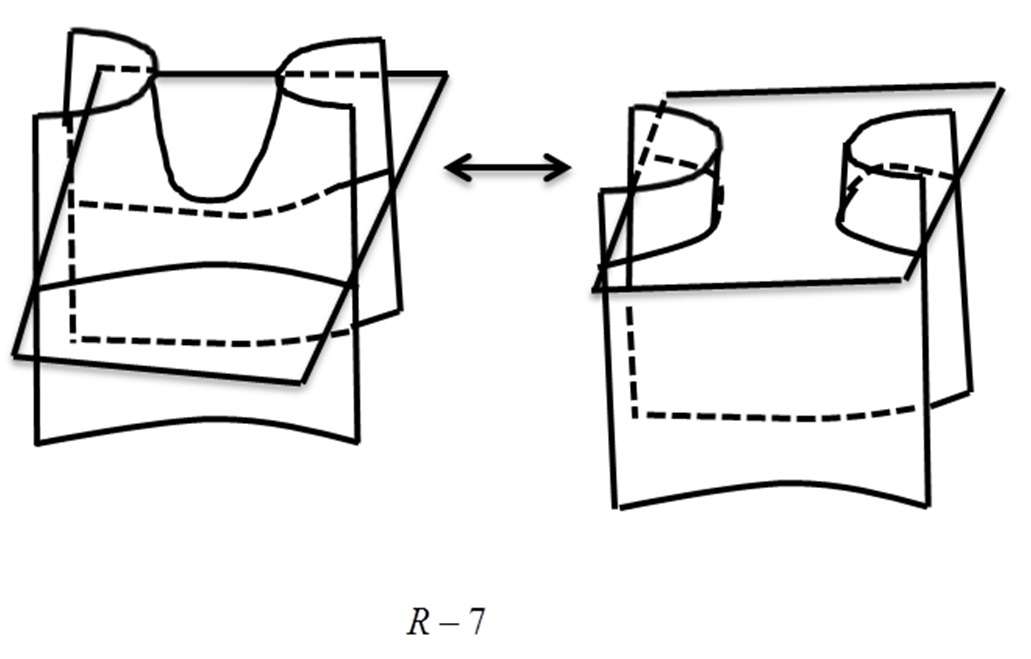}}
\caption{Roseman moves}
\label{2@}
   \end{figure}
We describe here Roseman move $R-7$. A disk $M$ embedded in a closure of one of complementary open regions of $\Delta$ is a \textit{descendent disk} if $M$ satisfies the following properties:
\begin{itemize}
\item[(i)] $\partial M =l_1 \cup l_2$, where $l_1$ and $l_2$ are two simple arcs in $\Delta$, 
\item[(ii)] $l_1 \cap l_2=\{c_1,c_2\}$, where $c_1$ and $c_2$ are double points and they are the endpoint of $l_1$ and $l_2$,
\item[(iii)] The endpoints of the pre-image of one of  $l_1$ or $l_2$ are both on upper decker set and both endpoints of the other one are on the lower decker set and
\item[(iv)] The double edges containing $c_1$ and $c_2$ have opposite orientation with respect to the orientation of the arc $l_1$ or $l_2$.
\end{itemize} 
In particular, if a descendent disk exists, then $R-7$ can be applied.\\  
There are six double edges incident to a triple point $T$. Such an edge is called  a \textit{$b/m$-, $b/t$- or $m/t$-edge} if it is the intersection of bottom
and middle, bottom and top, or middle and top sheets at $T$, respectively. 
 \begin{lem}[\cite{paper_2SatShim}]\label{R6}
Let $T$ be a triple point of a surface-knot diagram. Let $e$ be a $b/m$- or $m/t$-edge at $T$. If the other endpoint of $e$ is a branch point, then the triple point $T$ can be eliminated. 
\end{lem}
\begin{proof}
The sheet transverse to $e$ is the top or the bottom sheet at $T$. Hence, we can apply the Roseman move $R-6^-$ to move the branch point along $e$. As a result, $T$ will be eliminated.   
\end{proof} 
A surface-knot diagram of a surface-knot $F$ is \textit{t-minimal} if it is with minimal number of triple points for all possible diagrams of $F$.
\begin{lem}\label{loim}
Let $\Delta$ be a surface-knot diagram of a surface-knot $F$ and let $T$ be  a triple point of $\Delta$. Suppose $e$ is  a $b/t$- and $m/t$-edge (resp.  $b/m$-edge) at $T$. Also suppose that the closure of $e^U$ (resp. $e^L$) in the double decker set bounds a disk $D$ in $F$ such that the interior of $D$ does not meet the double decker set. Then $\Delta$ is not $t$-minimal.
\end{lem}
\begin{proof}
Let $B^3(D)$ be a 3-ball neighbourhood in 3-space of the projection of the disk $D$. Then clearly $B^3(D)$ must look like Figure \ref{R1}, where the horizontal sheet is the top (resp. bottom) sheet of $T$. There is an isotopy in $\mathbb{R}^4$ which changes Figure \ref{R1} to Figure \ref{R1dash}.  This isotopy is described in 3-space by attempting the move $R-5^+$ and then followed by the move $R-6^-$.  Hence, $T$ is eliminated and thus the lemma follows.
\end{proof}
\begin{figure}[H]
\captionsetup{font=scriptsize} 
  \centering
  
  \subfloat[]{\includegraphics[width=0.4\textwidth]{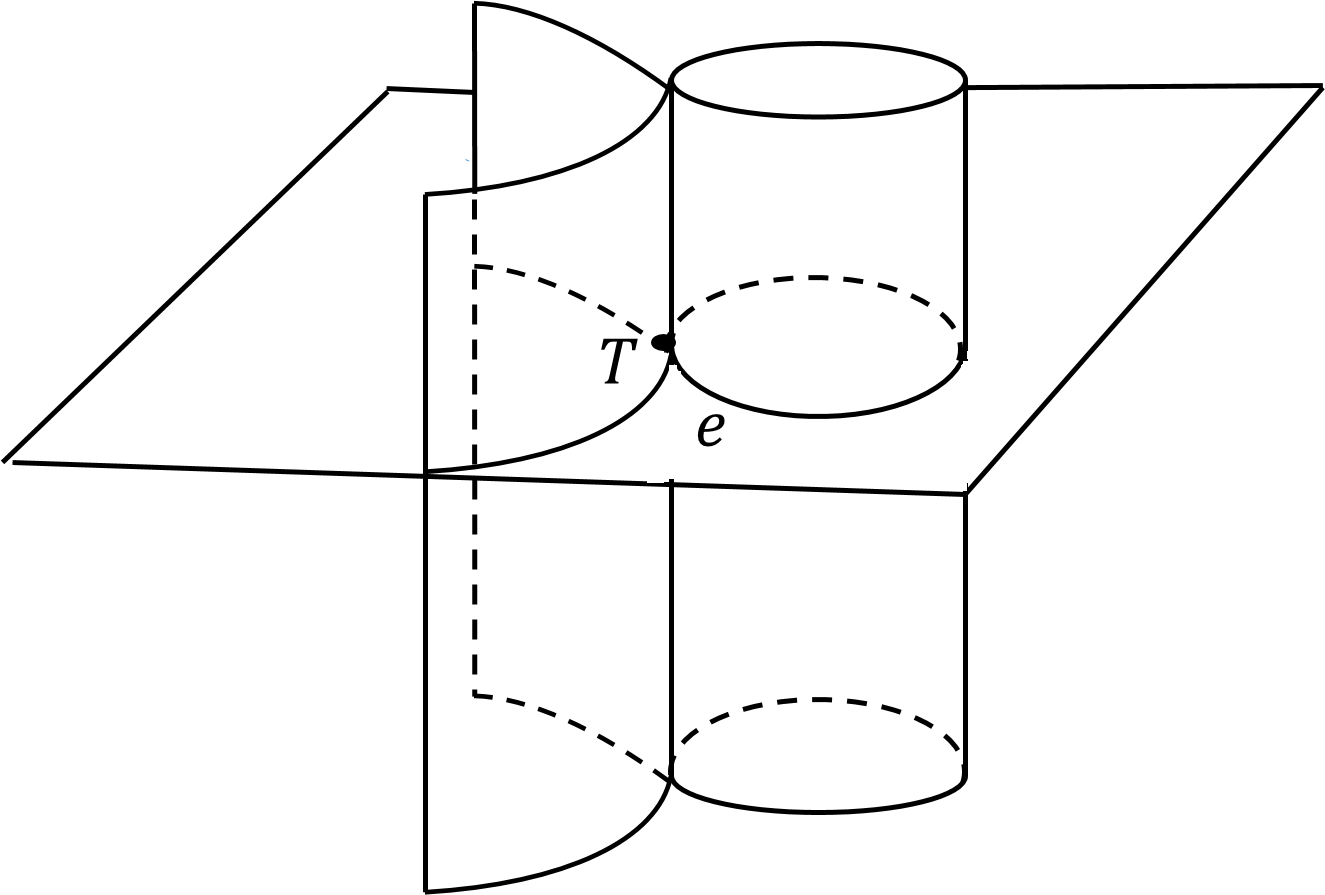}\label{R1}}
   \hfill
   \subfloat[]{\includegraphics[width=0.4\textwidth]{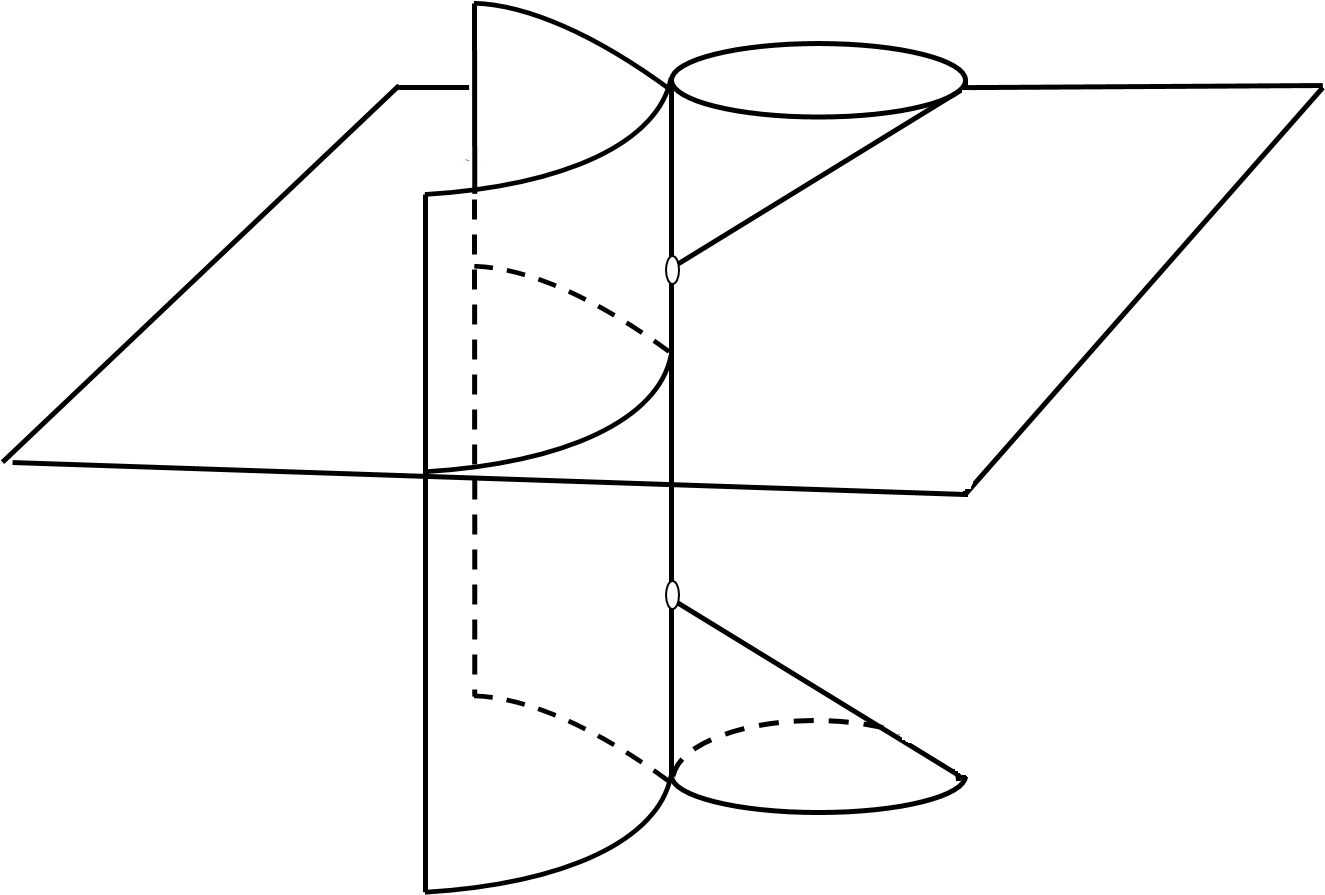}\label{R1dash}}
      
  \caption{The move $R-5^+$ is applied followed by the move $R-6^-$ to eliminate the triple point.}
 \label{can1}
\end{figure}

\begin{lem}\label{lo}
Let $\Delta$ be a surface-knot diagram of a surface-knot $F$ and let $T$ be  a triple point of $\Delta$. Suppose $e$ is  a $b/t$- and $m/t$-edge (resp.  $b/m$-edge) at $T$. Also suppose that the closure of $e^U$ (resp. $e^L$) in the double decker set bounds a disk $D$ in $F$ such that the projection of the interior of $D$ contains at most one triple point. Then the diagram $\Delta$ is not $t$-minimal.
\end{lem}
\begin{proof}
Let $D$ be the disk in $F$ bounded by the closure of  $e^U$ (resp. $e^L$).  Suppose
the interior of $D$ meets the double decker set of $\Delta$. If the interior of $D$ contains a trivial double
decker arc that is bounded by two branch points. Then obviously the corresponding trivial double point arc in $\Delta$ can be deleted by the move $R-4^-$. Thus, we may assume that there is
no trivial double decker arc inside D. Consider the following cases.
\begin{itemize}
\item[(i)] Suppose the interior of $D$ contains only disjoint simple closed curves. We may apply a deformation along the boundary of $D$ in $F$ so that the interior of the deformed disk does not intersect with the double decker set. For case the intersection is just single closed simple curve, this deformations is shown in Figure \ref{des} schematically by Roseman moves in
$\mathbb{R}^3$. On the other hand, suppose $D$ contains nested simple closed curves $S_1 \supset  S2\supset \dots  \supset S_n$.  Then we first remove the inner most curve $S_n$. This can be done by (i) applying a
finite sequence of the move $R-1^+$ so that the finger reaches $S_n$ and then (ii) applying the
move $R-7$ so that the modified $D$ does not include $S_n$. Next apply the move $R-7$ $n -1$
times to remove the curves $S_{n-1}, S_{n-2}, \dots S_1$ in order. Therefore, we may assume that $D$
does not contain any double decker set. Note that the isotopy in 3-space explained above never create a triple point. Now $\Delta$ is not t-minimal arises from Lemma 5.2.
\item[(ii)] Assume that the interior of $D$ contains a pre-image of a triple point $T_1$ in $\Delta$. Then there
are two double decker arcs 
$\gamma_1$ and $\gamma_2$  intersecting transversely at a crossing point in the
interior of $D$, that is the pre-image of $T_1$. By the assumption, the projection of the interior of $D$ does not contain
triple points other than $T_1$. Hence from Lemma \ref{even}, we obtain that 
$\gamma_1$ attains a branch
point on its boundary if and only if 
$\gamma_2$ does. If branch points are present at the boundaries
of 
$\gamma_1$ and 
$\gamma_2$, clearly we may eliminate the triple point $T_1$ by attempting the move $R-6^-$.
Thus,  $\Delta$ is not t-minimal in such a case. Suppose on the contrary that the boundaries of 
$\gamma_1$ and 
$\gamma_2$ never contain a branch point. Since $T_1$ is the only triple point inside the
projection of $D$, 
$\gamma_1$ and 
$\gamma_2$ form the figure eight shape inside $D$. The crossing point of
the figure eight shape corresponds to the triple point $T_1$. We see that the two double
decker loops of the figure eight shape bounds disks, $D_1$ and $D_2$, such that $D_1$ and $D_2$
might meet the double decker set at only simple closed curves. Now, we may apply the
deformation explained in (i) above to modify $D_1$ or $D_2$ so that the modified disk does not
meet the double decker set. Therefore, $T_1$ can be eliminated by Lemma \ref{loim}. Hence, $\Delta$ is
not t-minimal.

\end{itemize}
\end{proof} 
\begin{figure}[H]
\centering
\captionsetup{font=scriptsize}      
\mbox{\includegraphics[scale=0.3]{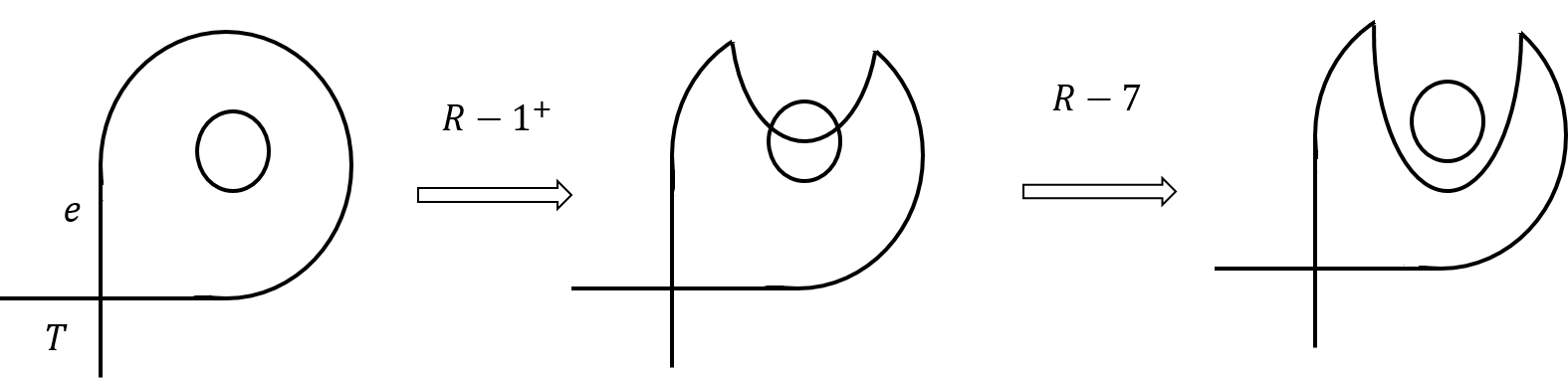}}
 \caption{The series of cross sections of the deformed diagram.}
 \label{des}
\end{figure}
\begin{lem}\label{c}
Let $\Delta$ be a surface-knot diagram of a surface-knot $F$. Let $T_1$ and $T_2$ be triple points in $\Delta$. Suppose $\Delta$ has double edges $e_1$ and $e_2$ such that 
\begin{itemize}
\item[(i)] $e_1$ is a $b/m$-edge at both $T_1$ and $T_2$ and  
\item[(ii)] $e_2$ is a $m/t$-edge at both $T_1$ and $T_2$.
\end{itemize}
Also suppose that the closure of $e_1^U \cup e_2^L$ in the double decker set bounds a disk $D$ in $F$ such that the projection of the interior of $D$ does not meet the double decker set. Then $\Delta$ is not $t$-minimal.
\end{lem}
\begin{proof}
Let $B^3(D)$ be a 3-ball neighbourhood in 3-space of the projection of the disk $D$. Then clearly $B^3(D)$ must look like Figure \ref{R2n}, where the vertical sheet is the middle sheet of both $T_1$ and $T_2$. There is an isotopy in $\mathbb{R}^4$ which changes Figure \ref{R2n} to Figure \ref{R2dash}.  This isotopy has an affect of elimination of the two triple points $T_1$ and $T_2$. Thus $\Delta$ is not $t$-minimal. The lemma follows.\end{proof}
\begin{figure}[H]
  \centering
  
  \subfloat[]{\includegraphics[width=0.4\textwidth]{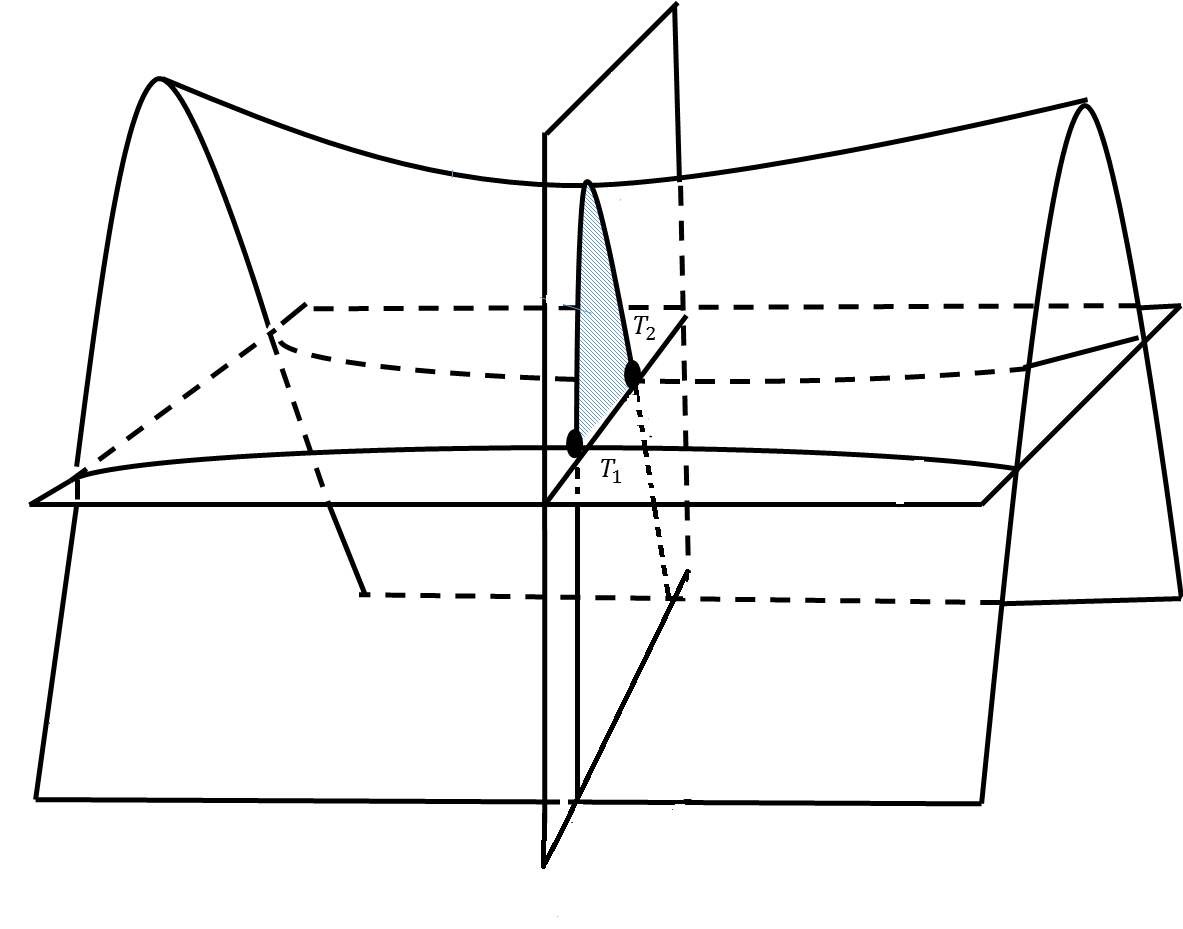}\label{R2n}}
   \hfill
   \subfloat[]{\includegraphics[width=0.4\textwidth]{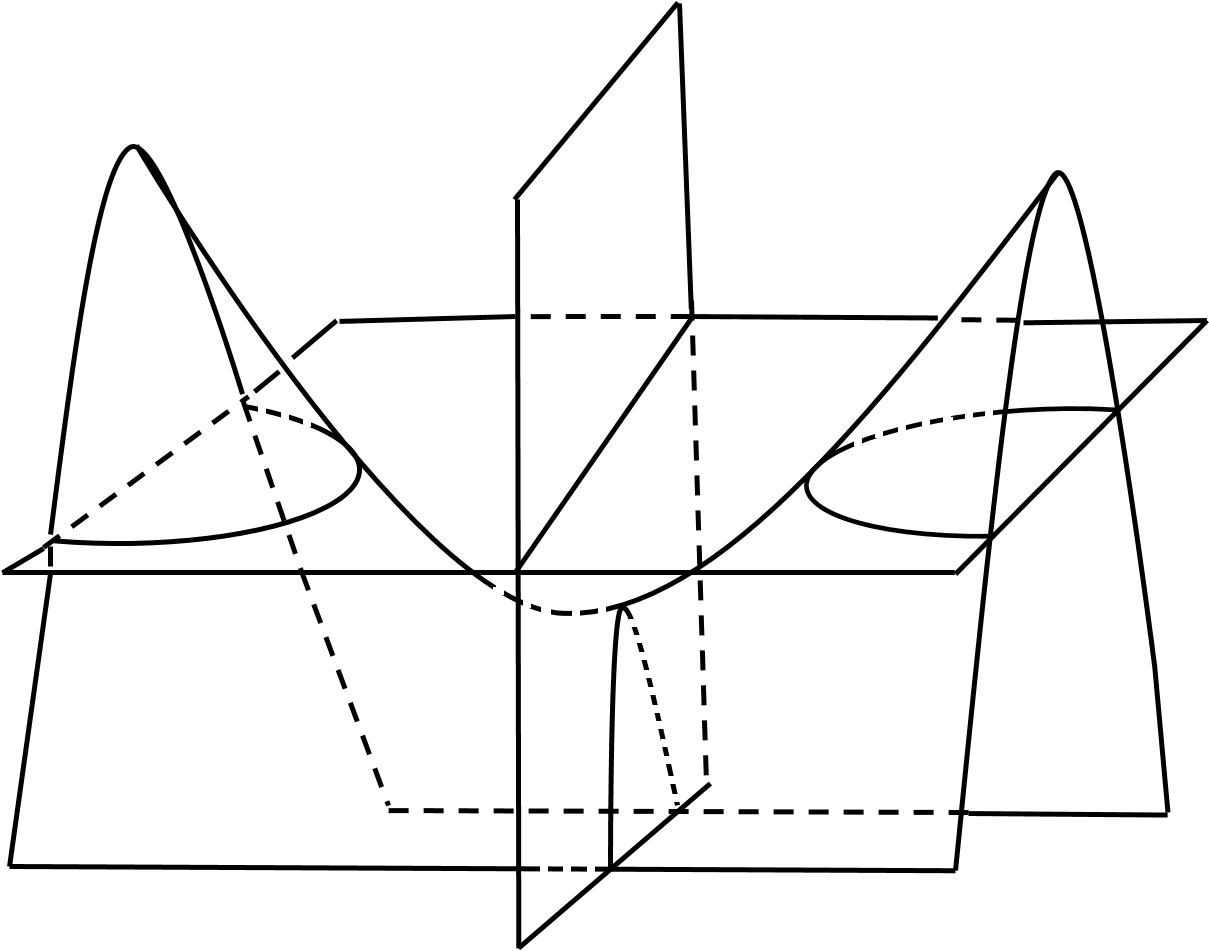}\label{R2dash}}
      
  \caption{The elimination of the pair of triple points. }
  \label{can1}
\end{figure}

\begin{lem}\label{c}
Let $\Delta$ be a surface-knot diagram of a surface-knot $F$. Let $T_1$ and $T_2$ be triple points in $\Delta$. Suppose $\Delta$ has double edges $e_1$ and $e_2$ such that 
\begin{itemize}
\item[(i)] $e_1$ is a $b/m$-edge at both $T_1$ and $T_2$ and 
\item[(ii)] $e_2$ is a $m/t$-edge at both $T_1$ and $T_2$.
\end{itemize}
Also suppose that the closure of $e_1^U \cup e_2^L$ in the double decker set bounds a disk $D$ in $F$ such that the projection of the interior of $D$ has at most one triple point. Then the diagram $\Delta$ is not $t$-minimal.
\end{lem}
\begin{proof}
By following the proof of Lemma \ref{lo}, we may assume that the interior of $D$ does not meet  the double decker set. Now the result follows from the previous lemma.
\end{proof}

\section{$t$-minimal surface-knot diagrams with three triple points}

It is proved in \cite{paper_2SatShim} that a surface-knot diagram with odd number of triple points has at least two double edges, each of which is bounded by a triple point and a branch point. In particular, a $t$-minimal diagram with three triple points (if it exists) satisfies the following. 
\begin{lem}[\cite{paper_2Sat3}]\label{s}
Assume that there is a surface-knot $F$ with $t(F)=3$. Let $\Delta$ be a $t$-minimal diagram of $F$ whose triple points are $T_1$, $T_2$ and $T_3$. Then after suitable changes of indices, the other boundary point of any of the $b/t$-edges at $T_1$ or $T_2$ is a triple point while the other endpoint of any of the $b/t$-edges at $T_3$ is a branch point. Moreover, $\lambda(T_1)=\lambda(T_2)=\lambda(T_3)$ and $\epsilon(T_1)= \epsilon(T_2)=-\epsilon(T_3)$. 
\end{lem}
\begin{proof}
The proof can be found in \cite{paper_2Sat3}.
\end{proof}
Let $\Delta$ be a surface-knot diagram whose description is  given in the above lemma.  By the orientation of the double edges coming to the triple points and the Alexander numbering assigned to the set of complementary connected regions $\mathbb{R}^3\setminus \vert \Delta \vert$, we list all possible connections of double edges that are incident to the three triple points of $\Delta$. This enumeration is shown in Table 1. Since the triple points $T_1$ and $T_2$ have same Alexander Numbering and signs, we may use the symbol $T_k$ to refer to either $T_1$ or $T_2$. The following example explains how table 1 should be read. The $b/m$-edges at $T_k$ can be joined to a $b/t$-edge at $T_k$ or a $b/m$-edge at $T_3$ or a $m/t$-edge at $T_3$.
\begin{center}
\begin{tabular}{|c||c|c|c|}
\cline{1-4}
 $T$ & \multicolumn{3}{ c| }{type of edges at $T$}  \\ \cline{2-4}
 & $b/m$-edge & $b/t$-edge & $m/t$-edge   \\  \hline \hline  
$T_k$ & $(T_k)_{b/t}$ , $(T_3)_{b/m}$ ,   $(T_3)_{m/t}$ & $(T_k)_{b/m}, (T_k)_{m/t} $ & $(T_k)_{b/t} , (T_3)_{b/m},  (T_3)_{m/t}$  \\ \hline
  
$T_3$ & $(T_k)_{b/m},  (T_k)_{m/t}$ & branch point & $(T_k)_{b/m},  (T_k)_{m/t}$
\\
\hline

\end{tabular}
\[
\text{Table 1}
\]
\end{center} 
We point out that the diagram $\Delta$ might contains some trivial double point circles.

\section{Lemmas}\label{section}
Throughout this section, we use the following notations.  Let $F$ be a surface-knot of genus one satisfying $t(F)=3$.  Let $
\Delta$ be a $t$-minimal surface-knot diagram of $F$ whose triple points are $T_1$, $T_2$ and $T_3$. By Lemma \ref{s}, we may assume that $\epsilon(T_1)=\epsilon(T_2)=-\epsilon(T_3)$ and $\lambda(T_1)=\lambda(T_2)=\lambda(T_3)$.  Also $\Delta$ has exactly one non-trivial double point arc consisting of a single triple point, $T_3$. Let this double point arc be denoted by $C$. 
\begin{lem}\label{TT}
The simple closed double decker curve $C^U \cup C^L$ is not homologous to zero in $F$.
\end{lem}
\begin{proof}
Let $C_1$ be the double point circle in $\Delta$ containing the $b/m$-edges at $T_3$ and let $C_2$ be the double point circle in $\Delta$ containing the $m/t$-edges at $T_3$. Consider the double decker set of $\Delta$.  Since $C_1^L$ is the lower decker curve and $C_2^U$ is the upper decker curve, it holds that  $C_1^L \neq C_2^U$. Let $C_{1_1}$ be the portion of $C_1$ in $\Delta$ that contains the $b/m$-edges at $T_3$ such that the pre-image
of $C_{1_1}$ in the lower sheet is a simple closed curve. In particular, the simple closed curve $C_{1_1}^L$ intersects $C^U \cup C^L$ transversely at a single crossing point, that is the pre-image of $T_3$ in the bottom sheet.
Therefore $C^U \cup C^L$ is not homologous to zero in $F$. The lemma follows.
\end{proof}
The double edge that is a $b/t$- and  $m/t$- (resp.  $b/m$-)edge at a triple point is called a \textit{double point loop}.
\begin{lem}\label{loop}
There is no double point loop in $\Delta$.
\end{lem}
\begin{proof}
Since the $b/t$-edges at $T_3$ are ended with branch points, $\Delta$ has no double point loop based at $T_3$. Suppose $e$ is a double point loop based at $T_1$ (resp. $T_2$) such that $e$ is a $b/t$-  and $m/t$-edge at $T_1$ (resp. $T_2$). Let $C_1$ be  the double point circle in $\Delta$ containing the $b/m$-edges at $T_3$. Let $C_{1_1}$ be the portion of $C_1$ that contains the $b/m$-edges at $T_3$ such that the pre-image of $C_{1_1}$ in the lower sheet is a simple closed curve. Consider the double decker set of $\Delta$ in $F$. The simple closed double decker curves $C^U \cup C^L$ and $C_{1_1}^L$ intersect transversely once in $F$. This implies that $[C^U \cup C^L]$ and $[C_{1_1}^L]$ represent the two distinct generators of the first homology group of $F$, $H_1(F)$. The loop $\overline{e^U}$ does not meet any of the generators and hence the subset of  $F$ enclosed  by $\overline{e^U}$ is homotopic to a disk, denoted by $D$.  Because $\Delta$ has only three triple points, the projection of the interior of $D$ can have at most one triple point. In particular, the assumption of Lemma \ref{lo} is satisfied and thus $\Delta$ is not $t$-minimal. This is a contradiction.  An analogous argument can be used to show the other case of the double point loop where $e$ is a $b/t$-  and $b/m$-edge at $T_1$ (resp. $T_2$).  
\end{proof}
\begin{lem}\label{descendent}
For $\Delta$, there exists no double edges $e_1$ and $e_2$ such that:  
\begin{itemize}
\item[(i)] $e_1$ is a $b/m$-edge at $T_3$ and  $b/m$-edge at $T_1$ (resp. $T_2$) and
\item[(ii)] $e_2$ is a $m/t$-edge at $T_3$ and  $m/t$-edge at $T_1$ (resp. $T_2$).
\end{itemize}
\end{lem}
\begin{proof}
Let $C_1$ be the double point circle in $\Delta$ containing the $b/m$-edges at $T_3$. Let $C_{1_1}$ be the portion of $C_1$ that contains the $b/m$-edges at $T_3$ such that the pre-image of $C_{1_1}$ in the lower sheet is a simple closed curve. It is shown in the previous lemma that $[C^U \cup C^L]$ and $[C_{1_1}^L]$ represent the two distinct generators of the first homology group of $F$, $H_1(F)$. Suppose for the sake of a contradiction that $\overline{e_1^U} \cup \overline{e_2^L}$ is not homologous to zero in $F$. Let the open sub-surface enclosed by $\overline{e_1^U} \cup \overline{e_2^L}$ be denoted by $F_1$. It is very easy to see that  $C^U \cup C^L$ lies in the complement of  $F_1$ in $F$. In fact,  $F_1$ does not meet the subsets of $F$ bounded by $C^U \cup C^L$ and $C_{1_1}^L$. Since $[C^U \cup C^L] \neq [C_{1_1}^L]\neq [0]$ in $H_1(F)$, we obtain that $F_1$ is homotopic to an open disk and that  $\overline{e_1^U} \cup \overline{e_2^L}$ is homologous to zero in $F$. Now, Lemma \ref{c} implies that $\Delta$ is not $t$-minimal. This is a contradiction.
\end{proof}

For triple points $T_i$ and $T_j$ in $\Delta$, an edge joins an $a/b$-edge at $T_i$ and $c/d$-edge at $T_j$ , the edge
will be denoted by $e_{ij}(a/b,c/d)$, $ (1 \leq  i, j \leq 3)$.
\begin{lem}\label{Tsukasa}
Let $C_1$ be a double point circle defined by
\[
C_1 = e_{32}(b/m, b/m) \cup e_{23}(b/m, m/t) \cup e_{31}(m/t, b/m) \cup e_{13}(b/m, b/m).
\]
The double point circle $C_1$ does not exist in the $t$-minimal surface-knot diagram $\Delta$.
\end{lem}
\begin{proof}
We shall show that there exists a surface-knot diagram having $C_{1}$ but the diagram is not $t$-minimal.

The remaining double edges at $T_{1}$ and $T_{2}$ form two double point circles, $C_{2}$ and $C_{3}$, such
that 
\begin{align*}
	C_{2} &= e_{12}(b/t, m/t)\cup e_{21}(m/t, b/t) \\
	C_{3} &= e_{12}(m/t, b/t)\cup e_{21}(b/t, m/t) 
\end{align*} 
Note that we avoid the connections which give a double point loop.

The connection of double point curves is depicted in Figure \ref{631}.
\begin{figure}[H]
\centering
\captionsetup{font=scriptsize}      
\mbox{\includegraphics[scale=0.4]{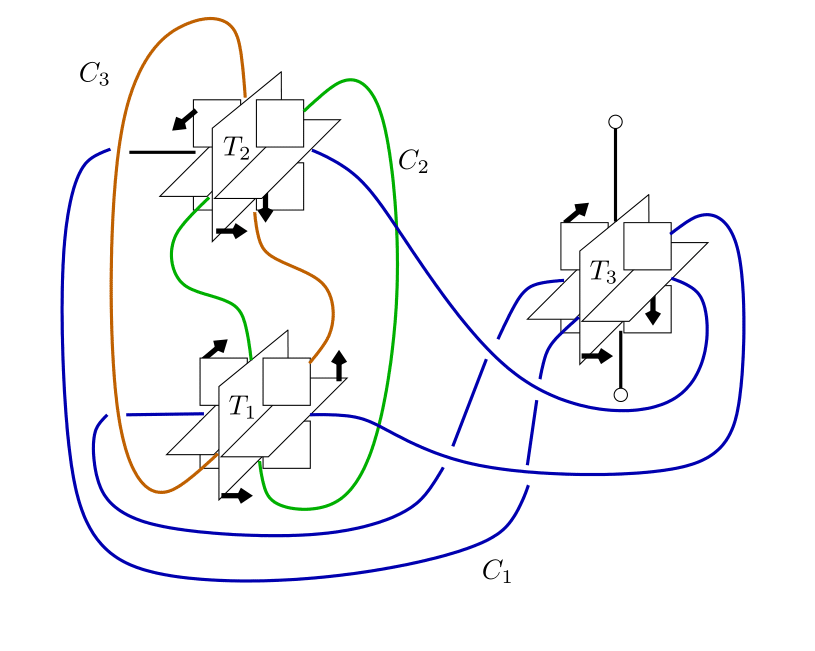}}
\caption{The double point curves between three triple points.}\label{631}
\end{figure}
The following describes the construction of the surface-knot diagram containing the double point curves depicted in Figure \ref{631}. 

Let $T_{i}^{X}$ denote the pre-image of the triple point $T_{i}$ $i=1, 2, 3$ on the $X$-sheet, where 
$X$ represents the top or middle or bottom.
The double decker set is depicted in Figure \ref{632}, in which the left rectangle represents 
the torus obtained by pasting the opposite sides in usual manner.

Let $E_{1}$ and $E_{2}$ be connected components of the double decker set such that
$E_{1}$ contains $T_{3}^{T}$ and $E_{2}$ contains $T_{1}^{T}$ and $T_{2}^{T}$.

The connected double decker set $E_{1}$ can be drawn as the left diagram in Figure \ref{632}.
\begin{figure}[H]
\centering
\captionsetup{font=scriptsize}      
\mbox{\includegraphics[scale=0.5]{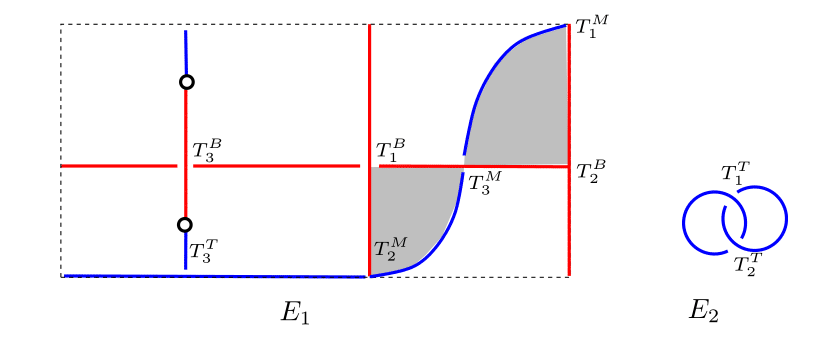}}
\caption{The double decker sets $E_{1}$ and $E_{2}$ on the rectangle.}\label{632}  
\end{figure}
The complement of the double decker set $E_{1}$ consists of seven open disks.  
$E_{2}$ should be placed in the complement of $E_{1}$. 
Let $D$ be the right half of the rectangle in Figure \ref{632} containing $T_{3}^{M}$ (see Figure \ref{632}). 
It is easy to see that $E_{2}$ is not on $D$. 

The closure of $D$ in $F$ forms an annulus and it contains two triangles, each of which has $\{ T_{1}^{B}, T_{2}^{M}, T_{3}^{M}\}$ and $\{T_{1}^{M}, T_{2}^{B}, T_{3}^{M}\}$ as its vertices respectively.
The annulus has two boundary circles which are projected on the component circles
of $E_{2}$ (see Figure \ref{633}).

\begin{figure}[H]
\centering
\captionsetup{font=scriptsize}      
\mbox{\includegraphics[scale=0.4]{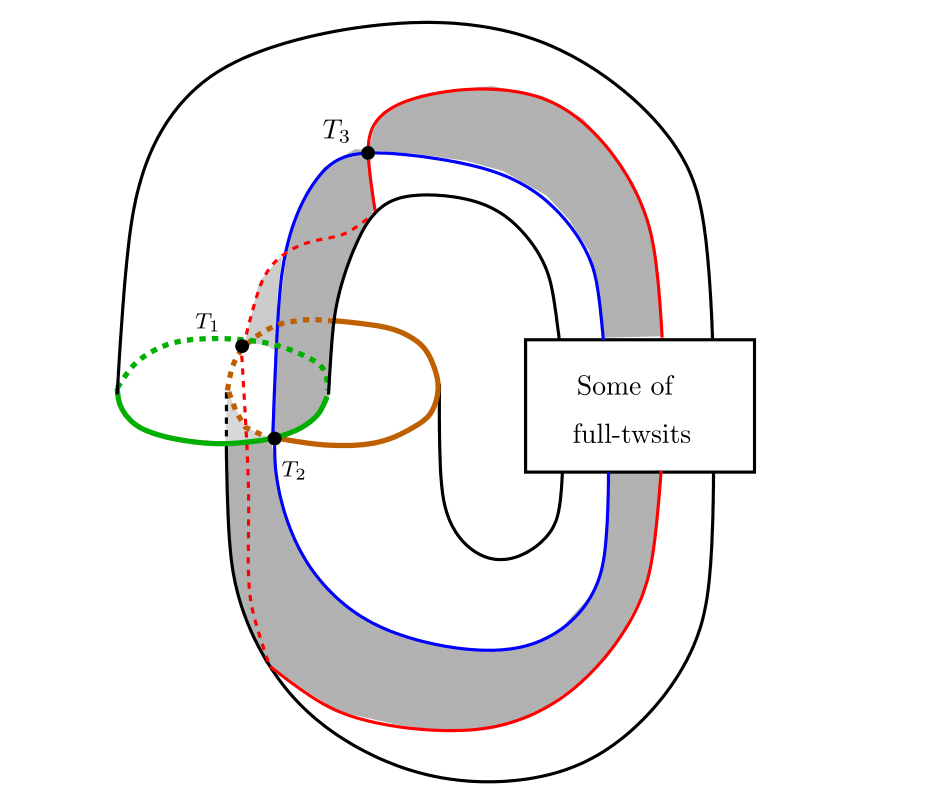}}
\caption{A possible projection of $D$ under the projection is depicted. 
Here, the shadow areas are the shaded triangles in Figure \ref{632}.}\label{633}  
\end{figure}
Therefore, in $\mathbb{R}^{3}$, the annulus determines the positions of $T_{1}$, $T_{2}$ $T_{3}$ up to 
homomorphism on the annulus. 
Also these triangles determine the double curves connecting $T_{3}$ and $\{T_{1}, T_{2}\}$.

Let $N$ be the closure of the image of a thin neighbourhood of $E_{1}$ under the projection.  
Following the double curves along the annulus, $N$ is uniquely determined up to homeomorphism of $N$. 
Once $N$ is determined, then the seven disks are uniquely determined up to isotopy in $\mathbb{R}^{3}$.
Then the position of $E_{2}$ is uniquely determined.
Therefore, the constructed diagram, if exists, is uniquely determined up to homeomorphism of the diagram. 

The surface-knot diagram containing the double point curves in Figure \ref{631} is depicted in
Figure \ref{634}. 

\begin{figure}[H]
\centering
\captionsetup{font=scriptsize}      
\mbox{\includegraphics[scale=0.5]{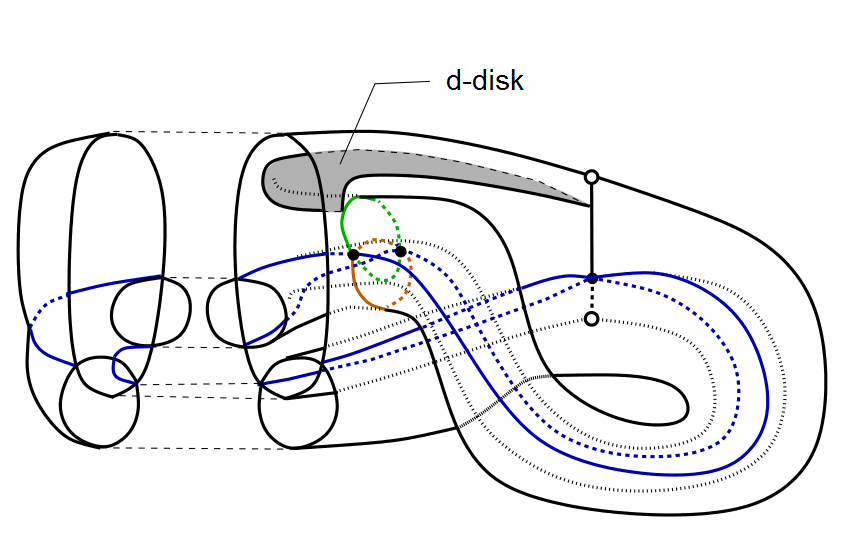}}
\caption{A possible surface-knot diagram corresponding to the double decker set in Figure \ref{632}.
The black dots are triple points. The white dots are branch points.}\label{634}
\end{figure}

There exists a descendent disk in the diagram (see Figure \ref{634}).
Applying the $R-7$ move along the descendent disk, one of the branch points joins to the $m/t$-branch
at $T_{2}$. Therefore, the triple point $T_{2}$ can be eliminated. 
This implies that the surface-knot diagram is not $t$-minimal.   

This completes the proof.

\end{proof}

\begin{lem}\label{c1}
There are at least two non-trivial double point circles in $\Delta$. 
\end{lem}
\begin{proof}
Suppose $\Delta$ has only one non-trivial double point circle. By excluding the cases of Lemma \ref{loop} and Lemma \ref{descendent}, we obtain only one possible double point circle that joins the double edges at the three triple points of $\Delta$. This double point circle is denoted by $C_1$ and can be described as follows:
\[
\begin{split}
C_1&=e_{21}(m/t,b/t) \cup  e_{12}(b/t,b/m) \cup e_{23}(b/m,m/t) \cup e_{31}(m/t,b/m) \\
&\quad \cup e_{12}(b/m,b/t) \cup e_{21}(b/t,m/t) \cup e_{13}(m/t,b/m) \cup e_{32}(b/m,m/t). 
\end{split}
\]
Consider the double decker set of $\Delta$. The closure of  $ \big( e_{13}(m/t,b/m)\big)^U \cup \big(e_{32}(b/m,m/t) \big)^U \cup  \big(e_{21}(m/t,b/t)\big)^U $ is a simple closed curve in $F$, let this curve be denoted by $\alpha$. Also the closure of $\big(e_{12}(b/t,b/m)\big)^L \cup \big( e_{23}(b/m,m/t)\big)^L \cup \big(e_{31}(m/t,b/m) \big)^L $ forms a simple closed curve, denoted by $\beta$. In particular, $\alpha$ and $\beta$ intersect transversely at  a single crossing point, $T_3^M$. It follows that $[\alpha]$ and $[\beta]$ represent the two distinct generators of the first homology group of $F$, $H_1(F)$. The simple closed curve $C^U \cup C^L$ does not meet any of the generators. Therefore, $C^U \cup C^L$ is homologous to zero in $F$, but this contradicts Lemma \ref{TT}.  
\end{proof}

\begin{lem}\label{circleb/m}
For $\Delta$, it is impossible to have a double point circle consisting of two double edges each of which is a $b/m$-edge at $T_2$ and $T_3$. 
\end{lem}
\begin{proof}
Let  $C_1=e_{23}(b/m,b/m) \cup  e_{32}(b/m,b/m)$ be the double point circle joining the $b/m$-edges at $T_2$ with the $b/m$-edges at $T_3$. Let $e$ be a $m/t$-edge at $T_3$. From Table 1, $e$ can be connected to (i) a $m/t$-edge at $T_2$ or (ii) a $m/t$-edge at $T_1$ or (iii) a $b/m$-edge at $T_1$. We show that any of these cases is impossible by the following.
\begin{itemize}
\item[ Case (i)] The assumption of Lemma \ref{descendent} is satisfied. Thus this case cannot occur.
\item [Case (ii)] In this case $e$ is a $m/t$-edge at both $T_3$ and $T_1$. Let $e_1$ be the $m/t$-edge at $T_1$ that is opposite to $e$. From Lemma \ref{loop},  the other boundary point of $e_1$ cannot be $T_1$. In particular,  the other boundary point of $e_1$ can be either $T_3$ or $T_2$. In the former case,  $\Delta$ has a double point circle, denoted by $C_2$, such that $C_2$ is defined by 
 \begin{equation}
 \renewcommand{\theequation}{ii-1}
 C_2= e_{31}(m/t,m/t) \cup  e_{13}(m/t,m/t).
\end{equation}
Suppose the latter case where the other boundary point of $e_1$ is $T_2$. From Table 1, $e_1$ is a $b/t$-edge at $T_2$. Let $e_2$ be the $b/t$-edge at $T_2$ that is opposite to $e_1$. From Lemma \ref{loop} and Table 1, we see that the other endpoint of $e_2$ is a $b/m$-edge at $T_1$ and that the other $b/m$-edge at $T_1$ is joined to a $m/t$-edge at $T_3$. We obtain the double point circle $e_{31}(m/t,m/t) \cup e_{12}(m/t,b/t) \cup  e_{21}(b/t,b/m) \cup e_{13}(b/m,m/t)$. The remaining double edges at the triple points of $\Delta$ form the double point circle $C_2$ such that 
\begin{equation}
 \renewcommand{\theequation}{ii-2}
 C_2=e_{12}(b/t,m/t) \cup  e_{21}(m/t,b/t).
\end{equation}
In both cases (ii-1) and (ii-2) above, consider the double decker set of $\Delta$ in $F$. We see that the closed double decker curves $C_1^U$ and $C_2^L$ meet transversely at a single crossing point. It follows that $[C_1^U]$ and $[C_2^L]$ represent the two distinct generators of the first homology group of $F$. The closed curve $C^U \cup C^L$ does not meet any of the generators. Therefore, $C^U \cup C^L$ is homologous to zero in $F$ which contradicts Lemma \ref{TT}.
  \item[Case(iii)] By following the similar way we did in the previous case to connect the double edges at the triple points, we obtain that $\Delta$ attains a double point circle which can be described by 
\begin{equation}
 \renewcommand{\theequation}{iii-1}
 e_{31}(m/t,b/m) \cup e_{13}(b/m,m/t)     
 \end{equation}
  or by
  \begin{equation}
 \renewcommand{\theequation}{iii-2}
  e_{31}(m/t,b/m) \cup e_{12}(b/m,b/t) \cup e_{21}(b/t,m/t) \cup e_{13}(m/t,m/t).  
  \end{equation}
It is not hard to see that in both cases (iii-1) and (iii-2), $\Delta$ has a double point circle $C_2$ such that 
\[C_2= e_{12}(b/t,m/t) \cup  e_{21}(m/t,b/t).
\] The conclusion now follows from the sub-case (ii-2) above.  
     
\end{itemize}   
\end{proof}
\begin{lem} \label{b/m}
For $\Delta$, it is impossible to have a double point circle consisting of two double edges each of which is a $b/m$-edge at $T_1$ and $T_3$. 
\end{lem}
\begin{proof}
Since $T_1$ and $T_2$ have same Alexander Numbering and signs, the proof of this lemma is analogous to the proof of the previous lemma.
\end{proof}
\begin{lem}\label{T1}
Let $e_1$ and $e_2$ be the $b/m$-edges at $T_2$ in $\Delta$. Suppose the other endpoint of $e_1$ is the triple point $T_3$ such that $e_1$ is a $b/m$-edge at $T_3$. The connection of double edges in $\Delta$ in which the other endpoint of $e_2$ is the triple point $T_1$ can not occur.
\end{lem}
\begin{proof}
Suppose for the sake of a contradiction that the other endpoint of $e_2$ is $T_1$. From Table 1, $e_2$ is a $b/t$-edge at $T_1$.
Let $e_3$ be the $b/t$-edge at $T_1$ that is opposite to $e_2$. From Lemma \ref{loop} and Table 1, $e_3$ is a $m/t$-edge at $T_2$. Let $e_4$ be the $m/t$-edge at $T_2$ that is opposite to $e_3$. The other endpoint of $e_4$ is $T_3$ by Table 1 and Lemma \ref{loop}. If $e_4$ is a $m/t$-edge at $T_3$, then the assumption of Lemma \ref{descendent} is satisfied and so we get a contradiction. On the other hand if $e_4$ is a $b/m$-edge at $T_3$, the double point circle $C_1$ is obtained such that 
\[
  C_1= e_{32}(b/m,b/m) \cup e_{21}(b/m,b/t) \cup e_{12}(b/t,m/t) \cup e_{23}(m/t,b/m). 
  \] 
Consider the cases below.
  \begin{itemize}
  \item[(i)] $\Delta$ has exactly two non-trivial double point circles. In this case, the other non-trivial double point circle, denoted by $C_2$, can be described as 
  \[ 
  C_2= e_{31}(m/t,m/t) \cup e_{12}(m/t,b/t) \cup e_{21}(b/t,b/m) \cup e_{13}(b/m,m/t).
  \]
In the double decker set of $\Delta$, the simple closed  curves $C_2^L$ and $C_1^U$ meet at exactly one transverse crossing point in $F$. Thus $[C_2^L]$ and $[C_1^U]$ represent the two distinct generators of the first homology group of $F$, $H_1(F)$. But the curve $C^U \cup C^L$ does not intersect any of the generators which implies that $[C^U \cup C^L]=[0]$ in $H_1(F)$. This contradicts Lemma \ref{TT}.  
   
  \item[(ii)] $\Delta$ has more than two non-trivial double point circles. Then with $C_1$, $\Delta$ has two more double point circles, $C_2$ and $C_3$. In particular, $C_2$ and $C_3$ can be defined by
  \begin{equation} 
\begin{array}{ll} 
 C_2= e_{31}(m/t,b/m) \cup e_{13}(b/m,m/t), \\
 C_3=  e_{21}(b/t,m/t) \cup e_{12}(m/t,b/t).
 \end{array} \
 \renewcommand{\theequation}{ii-1}
\end{equation}
  Or by 
  \begin{equation} 
\begin{array}{ll} 
 C_2= e_{31}(m/t,m/t) \cup e_{13}(m/t,m/t), \\
 C_3=  e_{21}(b/t,b/m) \cup e_{12}(b/m,b/t).
 \end{array} \
  \renewcommand{\theequation}{ii-2}
\end{equation}
In both cases (ii-1) and (ii-2), the simple closed  curves $C_2^L$ and $C_1^U$ meet at exactly one transverse crossing point in $F$. Thus $[C_2^L]$ and $[C_1^U]$ represent the two distinct generators of the first homology group of $F$, $H_1(F)$. But the  curve $C^U \cup C^L$ does not intersect any of the generators which implies that $[C^U \cup C^L]=[0]$ in $H_1(F)$. This contradicts Lemma \ref{TT}. 
 \end{itemize}

\end{proof}
\begin{lem}\label{T3}
Let $e_1$ and $e_2$ be the $b/m$-edges at $T_2$ in $\Delta$. Suppose the other endpoint of $e_1$ is the triple point $T_3$ such that $e_1$ is a $b/m$-edge at $T_3$. The connection of double edges in $\Delta$ in which the other endpoint of $e_2$ is the triple point $T_3$ can not occur.
\end{lem}
\begin{proof}
Assume that the other endpoint of $e_2$ is $T_3$. From Lemma \ref{circleb/m}, $e_2$ cannot be a $b/m$-edge at $T_3$. In particular,  $e_2$ is a $m/t$-edge at $T_3$.
Let $e_3$ be a $m/t$-edge at $T_3$ that is opposite to $e_2$. From Lemma \ref{descendent}, the other boundary point of $e_3$ cannot be $T_2$. From Table 1, $e_3$ is either (i) a $m/t$-edge at $T_1$ or (ii) a $b/m$-edge at $T_1$.
\begin{itemize}
\item[Case (i)]  Let $e_4$ be a $m/t$-edge at $T_1$ that is opposite to $e_3$. If the other boundary point of $e_4$ is $T_1$ then we obtain a double point loop based at $T_1$ in $ \Delta$. Such a case cannot happen by Lemma \ref{loop}. Also the other endpoint of $e_4$ cannot be $T_2$ (For if the other endpoint of $e_4$ is $T_2$, then $e_4$ is a $b/t$-edge at $T_2$ by Table 1. Let $e_5$ be the $b/t$-edge at $T_2$ that is opposite to $e_4$. In particular, $e_5$ is a $b/m$-edge at $T_1$. Let $e_6$ be the $b/m$-edge at $T_1$ that is opposite to $e_5$. If the other endpoint of $e_6$ is $T_3$, then the assumption of Lemma \ref{descendent} between $T_3$ and $T_1$ is satisfied and so we get a contradiction. Hence the other endpoint of $e_6$ is $T_1$. But in such a case we get a double point loop and so this case is also impossible by Lemma \ref{loop}). We obtain that the other boundary point of $e_4$ is $T_3$ where $e_4$ is a $b/m$-edge at $T_3$. A double point circle $C_1$ is obtained such that
\[C_1= e_{32}(b/m,b/m) \cup e_{23}(b/m,m/t) \cup e_{31}(m/t,m/t) \cup e_{13}(m/t,b/m). 
\]
The remaining double edges that are coming to the triple points of $\Delta$ form two double point circles, $C_2$ and $C_3$ such that 
\begin{equation} 
\begin{array}{ll} 
 C_2= e_{12}(b/m,b/t) \cup e_{21}(b/t,b/m),\\
 C_3=  e_{12}(b/t,m/t) \cup e_{21}(m/t,b/t).
\end{array} \
 \nonumber
\end{equation}
We show that this connection cannot occur by the following. The simple closed curves $C_2^U$ and $C_3^U$ meet at exactly one transverse crossing point in $F$. Thus $[C_2^U]$ and $[C_3^U]$ represent the two distinct generators of the first homology group of $F$, $H_1(F)$. But the  curve $C^U \cup C^L$ does not intersect any of the generators which implies that $[C^U \cup C^L]=[0]$ in $H_1(F)$. This contradicts Lemma \ref{TT}. 
\item[Case (ii)] Assume that $e_3$ is a $b/m$-edge at $T_1$. Let $e_4$ be the $b/m$-edge at $T_1$ that is opposite to $e_3$. From Lemma \ref{loop} and Table 1, we have two possibilities for the other endpoint of $e_4$: Either $e_4$ is a $b/m$-edge at $T_3$ or $e_4$ is a $b/t$-edge at $T_2$. The former case cannot happen by Lemme \ref{Tsukasa}. Consider the latter case. By avoiding the case in Lemma \ref{loop}, we obtain a double point circle $C_1$ in $\Delta$ which can be described by 
\[
\begin{split}
C_1&=e_{32}(b/m,b/m) \cup e_{23}(b/m,m/t) \cup e_{31}(m/t,b/m) \\
&\quad \cup e_{12}(b/m,b/t) \cup e_{21}(b/t,m/t)\cup e_{13}(m/t,b/m) .  
\end{split}
\]
 So, the remaining double edges at $T_1$ and $T_2$ form a double point circle, $C_2$, such that
 \[C_2= e_{12}(b/t,m/t) \cup e_{21}(m/t,b/t).\]
In double decker set of $\Delta$, the closure of $\big(e_{23}(b/m,m/t)\big)^L \cup \big(e_{31}(m/t,b/m)\big)^L \cup \big(e_{12}(b/m,b/t)\big)^L $ is a simple closed curve in $F$ that meets the simple closed lower decker curve $C_2^L$ at a single transverse crossing point. Therefore, these two curves are not homologous to zero and in fact they represent the two distinct generators of the first homology group of $F$, $H_1(F)$. The curve $C^U \cup C^L$ has empty intersection with these two generators and thus $C^U \cup C^L$ is homologous to zero. This contradicts Lemma \ref{TT}.

\end{itemize} 
\end{proof}
\begin{prop}\label{b/mb/m}
For $\Delta$, it is impossible to have a double edge $e$ such that $e$ is a $b/m$-edge at both $T_3$ and $T_2$ (resp. $T_1$).
\end{prop}
\begin{proof}
Suppose $e$ is a $b/m$-edge at both $T_3$ and $T_2$. Let $e_1$ be the $b/m$-edge at $T_2$ that is opposite to $e$. Lemmas \ref{loop}, \ref{T1} and \ref{T3} imply that the other endpoint of $e_1$ cannot be $T_1$ or $T_2$ or $T_3$. Since $\Delta$ has only three triple points and by Lemma \ref{R6}, we conclude that this connection is impossible. Since $T_1$ and $T_2$ have same Alexander numbering and signs, the argument used above is also valid if $e$ is a $b/m$-edge at both $T_3$ and $T_1$. The proposition follows.   
\end{proof}
\begin{cor}\label{m/tm/t}
For $\Delta$, it is impossible to have a double edge $e$ such that $e$ is a $m/t$-edge at both $T_3$ and $T_2$ (resp. $T_1$).
\end{cor}
\begin{proof}
Let $e$ be a double edge in $\Delta$ such that $e$ is a $m/t$-edge at both $T_3$ and $T_2$. Let $\Delta_1$ be the mirror image of $\Delta$. Since $\Delta$ is a $t$-minimal diagram, so does $\Delta_1$. In fact, $\Delta_1$ has a double edge that is a $b/m$-edge at both $T_3$ and $T_2$ (that is the mirror image of $e$). By Proposition \ref{b/mb/m}, $\Delta_1$ is not  $t$-minimal, a contradiction. An analogous argument might be used to show the other case where $e$ is a $m/t$-edge at both $T_3$ and $T_1$.   
\end{proof}

\begin{lem}\label{circlet/m}
For the diagram $\Delta$, it is impossible to have a double point circle $C_1$ with \[C_1=e_{32}(b/m,m/t) \cup e_{23}(m/t,b/m).\]
\end{lem}
\begin{proof}
Let $e_1$ be a $m/t$-edge at $T_3$. From Table 1, the other boundary point of $e_1$ is either (i) $T_2$ or (ii) $T_1$. 
\begin{itemize}
\item[Case (i)] First assume that the other endpoint of $e_1$ is $T_2$. In particular, $e_1$ is a $b/m$-edge at $T_2$. Let $e_2$ be the $b/m$-edge at $T_2$ that is opposite to $e_1$. Consider the following cases.
\begin{itemize}
\item[(i-1)] Suppose the other endpoint of $e_2$ is $T_3$. Then $e_2$ is a $m/t$-edge at $T_3$. We get the double point circle $C_2= e_{32}(m/t,b/m) \cup  e_{23}(b/m,m/t)$. It is not difficult to see that $\Delta$ has a double point loop that is based at $T_1$. 
\item[(i-2)] Suppose the other endpoint of $e_2$ is $T_1$. From Table 1, $e_2$ is a $b/t$-edge at $T_1$. Let $e_3$ be the $b/t$-edge at $T_1$ that is opposite to $e_2$. The other endpoint of $e_3$ must be $T_1$ by Table 1 and this gives a double point loop in $\Delta$ based at $T_1$. 
\item[(i-3)] If the other endpoint of $e_2$ is $T_2$, then $\Delta$ has a double point loop based at $T_2$.
\end{itemize}
In all cases (i-1) and (i-2) and (i-3), $\Delta$ is not $t$-minimal by Lemma \ref{loop}. This is a contradiction. Hence, the other endpoint of $e_1$ cannot be $T_2$. 
\item[Case (ii)] Suppose on the contrary that the other boundary point of $e_1$ is $T_1$. By Corollary \ref{m/tm/t}, $e_1$ cannot be a $m/t$-edge at $T_1$. Therefore by Table 1, $e_1$ is a $b/m$-edge at $T_1$. Let $e_2$ be the $b/m$-edge at $T_1$ that is opposite to $e_1$. The other endpoint of $e_2$ cannot be $T_1$ by Lemma \ref{loop} and it is either (ii-1) $T_2$ or (ii-2) $T_3$.

\begin{itemize}
\item[(ii-1)] Suppose $e_2$ is bounded by $T_2$. In fact, $e_2$ is a $b/t$-edge at $T_2$. Let $e_3$ be the $b/t$-edge at $T_2$ that is opposite to $e_2$. Then $e_3$ is a $m/t$-edge at $T_1$. Let $e_4$ be the $m/t$-edge at $T_1$ that is opposite to $e_3$.  From Table 1, the other endpoint of $e_4$ can be $T_1$ or $T_3$.  In particular,  $e_4$ cannot be ended with  $T_1$ because in such a case $\Delta$ is not $t$-minimal by Lemma \ref{loop}. Also it is impossible for $e_4$ to be a $m/t$-edge at $T_3$ by Corollary \ref{m/tm/t}. Hence, this connection cannot occur.
\item[(ii-2)] Assume that $e_2$ is bounded by $T_3$. Then  $e_2$ is a $m/t$-edge at $T_3$. We obtain the double point circle $C_3=e_{31}(m/t,b/m) \cup  e_{13}(b/m,m/t)$. Consider the double decker set of $\Delta$ in $F$. The simple closed curves $C_1^U$ and $C_3^L$ meet at a single transverse crossing point. Therefore, $[C_1^U]$ and $[C_3^L]$ are the two distinct generators of the first homology group of $F$, $H_1(F)$. The curve $C^U \cup C^L$ does not meet any of the generators and thus $C^U \cup C^L$ is homologous to zero in $F$. This contradicts Lemma \ref{TT}. 
\end{itemize}
\end{itemize}   
\end{proof}
\begin{prop}\label{2}
For  $\Delta$, it is impossible to have a double edge $e$ such that $e$ is a $b/m$-edge at $T_3$ and a $m/t$-edge at  $T_2$ (resp. $T_1$).
\end{prop}
\begin{proof}
We prove the case where $e$ is a $b/m$-edge at $T_3$ and a $m/t$-edge at $T_2$. The other case where $e$ is a $b/m$-edge at $T_3$ and a $m/t$-edge at $T_1$ can be analogously shown. Let $e_1$ be the $m/t$-edge at $T_2$ that is opposite to $e$. If the other endpoint of $e_1$ is $T_2$, then we obtain a double point loop based at $T_2$ in $\Delta$. In such a case, $\Delta$ is not $t$-minimal by Lemma \ref{loop}, a contradiction.  From Corollary \ref{m/tm/t} and Lemma \ref{circlet/m}, $e_1$ is neither a $m/t$-edge at $T_3$ nor a $b/m$-edge at $T_3$. So the other endpoint of $e_1$ cannot be $T_3$. The remaining case to consider is that the other boundary point of $e_1$ is $T_1$. In this case, $e_1$ is a $b/t$-edge at $T_1$. Let $e_2$ be the $b/t$-edge at $T_1$ that is opposite to $e_1$. The double edge $e_2$ is a $b/m$-edge at $T_2$ . Let $e_3$ be the $b/m$-edge at $T_2$ that is opposite to $e_2$. Then $e_3$ is neither a $b/m$-edge at $T_3$ by Proposition \ref{b/mb/m} nor a $b/t$-edge at $T_2$ by Lemma \ref{loop}. Therefore, $e_3$ is a $m/t$-edge at $T_3$. Let $e_4$ be the $m/t$-edge at $T_3$ that is opposite to $e_3$. Corollary \ref{m/tm/t} and Table 1 imply that $e_4$ is a $b/m$-edge at $T_1$. From Proposition \ref{b/mb/m} and Table 1 , the $b/m$-edge at $T_1$ that is opposite to $e_4$ is a $b/t$-edge at $T_2$. Now it is not difficult to see that $\Delta$ has a unique non-trivial double point circle which contradicts Lemma \ref{c1}.   
\end{proof}
\section{The main result}
We are now ready to state and prove the main theorem.
\begin{thm}
There is no surface-knot of genus one with triple point number invariant equal to three.
\end{thm}
\begin{proof}

Suppose for the sake of a contradiction that $F$ is a surface-knot of genus one and satisfies $t(F)=3$. Let $\Delta$ be a surface-knot diagram of $F$ with three triple points whose triple points are $T_1$, $T_2$ and $T_3$. By Lemma \ref{s}, we may assume that $\epsilon(T_1)=\epsilon(T_2)=-\epsilon(T_3)$ and $\lambda(T_1)=\lambda(T_2)=\lambda(T_3)$. Let $e$ be a $b/m$-edge at $T_3$. Let $T_k$ refers to $T_1$ or to $T_2$. From Table 1, the other boundary point of $e$ is $T_k$ where $e$ is a $b/m$-edge at $T_k$ or a $m/t$-edge at $T_k$. But these connections contradict Propositions \ref{b/mb/m} and \ref{2}. The Theorem follows.
\end{proof}

In \cite{amtsu1}, we proved that for genus-one surface-knots, the triple point number is at least three. From this and the previous theorem, we have the following corollary.
\begin{cor}
For surface-knots of genus one, the triple point number $t(F)$ satisfies $t(F) \geq 4$.
\end{cor}

\begin{Acknowledgement}
The authors would like to thank Prof. Seiichi Kamada and Prof. Scott Carter for their valuable comments and discussions.
\end{Acknowledgement}

\noindent{\sc Tsukasa Yashiro}
\newline
{\it Independent Mathematical Institute, Japan.\\
e-mail} : {\verb|tyashiro@cobalt.net.tokai-u.jp|} \\ \\
\noindent{\sc Amal Al Kharusi}
\newline
{\it Pathway Programme, Muscat University, Oman\\
e-mail} : {\verb|amalalkharusi2@gmail.com|}

\begin{comment}
\bibliographystyle{amsplain}

\bibliography{paperf} 
 \end{comment}
\end{document}